\documentclass[12pt, reqno]{amsart}
\usepackage{amsmath, amsthm, amscd, amsfonts, amssymb, graphicx, color}
\usepackage[bookmarksnumbered, colorlinks, plainpages]{hyperref}

\newtheorem{theorem}{Theorem}[section]
\newtheorem{lemma}[theorem]{Lemma}
\newtheorem{proposition}[theorem]{Proposition}
\newtheorem{corollary}[theorem]{Corollary}
\theoremstyle{definition}

\theoremstyle{remark}
\newtheorem{remark}[theorem]{Remark}
\numberwithin{equation}{section}

\begin{document}
\setcounter{page}{1}

\title[Compact and weakly compact multipliers]{Compact and weakly compact multipliers on Fourier algebras of ultraspherical hypergroups}

\author[R. EsmailvandiM. Nemati ]{  Reza Esmailvandi$^1$ and Mehdi Nemati$^2$}
\address{$^{1}$Department of Mathematical Sciences,
     Isfahan Uinversity of Technology, Isfahan 84156-83111, Iran;}
\email{\textcolor[rgb]{0.00,0.00,0.84}
{ r.esmailvandi@math.iut.ac.ir}}
\address{$^{2}$ Department of Mathematical Sciences,
	Isfahan Uinversity of Technology,
	Isfahan 84156-83111, Iran;
	 \newline
	 School of Mathematics,
	 Institute for Research in Fundamental Sciences (IPM),
	 P.O. Box: 19395--5746, Tehran, Iran.}
\email{\textcolor[rgb]{0.00,0.00,0.84}{m.nemati@cc.iut.ac.ir}}





\begin{abstract}
 A locally compact group $ G $ is discrete if and only if the Fourier algebra $ A(G) $ has a non-zero (weakly) compact multiplier.  We partially extend this result to the setting of ultraspherical hypergroups. Let $H$ be an ultraspherical hypergroup  and let $A(H)$ denote the corresponding Fourier algebra. We will give several characterizations of discreteness of $ H $  in the terms of  the  algebraic properties of $A(H)$.  We also study Arens regularity of  closed ideals of $ A(H)$.
\end{abstract} \maketitle

\section{Introduction}
Let $G$ be a locally compact group. We let $ A(G) $ and $VN(G)$ denote the Fourier  and  group von Neumann algebras  of $G$ which is introduced by Eymard \cite{eym}. It is known from  \cite{lau1979}  by Lau  that the Fourier algebra $ A(G) $ has a non-zero (weakly) compact left multiplier if and  only if $ G $ is discrete. He also proved that for a discrete and amenable group $ G $, $ A(G)$ is precisely the  algebra of  all weakly compact multipliers on $ A(G)$.
In addition, Ghahramani and Lau in \cite{ghl} showed that $G$ is discrete if and only if the second dual of $A(G)$ equipped with the first Arens product has a (weakly) compact left   multiplier $T$ such that $\langle T(n) , \lambda(e)\rangle \neq 0 $ for some $ n \in A(G)^{**} $. 

In recent years, some classes  of hypergroups whose Fourier space forms
a Banach algebra under pointwise multiplication have been discovered \cite{am, mur1, mur2, verm}. Most notably,
this concerns the class of ultraspherical hypergroups,  which includes in particular all double coset hypergroups and
hence all orbit hypergroups. Since ultraspherical hypergroups are generalized versions of locally compact groups, it is natural to ask whether the above results hold in the ultraspherical hypergroup setting. One of the purpose of this paper is to provide an affirmative answer to this question.

In Section 2, after recalling
some background notations and definitions, we discuss some basic properties
of the Fourier algebra of  ultraspherical hypergroups that we shall need in establishing our
main results.

In section 3, we will give some characterizations for discreteness of the ultraspherical hypergroup $ H $  in terms of the existence of  minimal idempotents in $A(H)$. 

In section 4, we study (weakly) compact multipliers of $ A(H) $. As an application of this result, we show  that $ A(H) $  is an ideal in  $ VN(H)^{*} $  if an only if $ H $ is discrete.  We also prove that  $ H $ is discrete if and only if  there is  a weakly compact right (equivalently, left) multiplier $ T $ of $ VN(H)^*$ and $ m\in VN(H)^*$ such that $ T(m)\in A(H) $ and
$ \langle T(m),\lambda(\dot{e})\rangle\neq0 $.  

In the final section we study Arens regularity of  closed ideals of $ A(H)$. Furthermore, we give a characterization of $ H $ to be finite in  terms of weakly compact multipliers of $ A(H) $. 
\section{Preliminaries}
Let $H$ be an ultraspherical hypergroup associated to  a locally compact group
$G$
and a spherical
projector $\pi: C_c(G)\rightarrow C_c(G)$ which was introduced
and studied in \cite{mur2}. Let $A(H)$ denote the Fourier algebra corresponding to the hypergroup
$H$. A left Haar measure on $H$ is given by $\int_{H} f(\dot{x})d\dot{x}=\int_G f(p(x))dx$,  $f\in C_{c}(H)$, where  $p: G \rightarrow H$ is the quotient map. Recall that the Fourier algebra
$A(H)$ is semisimple, regular and Tauberian\cite[Theorem 3.13]{mur2}. As in the group case, let $\lambda$ also denote the left regular representation of $H$ on $L^2(H)$ given by 
$$
\lambda(\dot{x})(f)(\dot{y})=f(\check{\dot{x}}\ast\dot{y})\quad (\dot{x},\dot{y}\in H, f\in L^2(H))
$$
This can be extended to $L^1(H)$ by $\lambda(f)(g)=f*g$ for all $f\in L^1(H)$ and $g\in L^2(H)$. Let $C^*_{\lambda}(H)$ denote the completion of $\lambda(L^1(H))$
in $B(L^2(H))$ which is called the reduced $C^*$-algebra of $H$. The von Neumann algebra generated
by $\{\lambda(\dot{x}): \dot{x}\in H\}$ is called the von Neumann algebra of $H$, and is denoted by $VN(H)$. Note that $VN(H)$ is isometrically isomorphic to the dual of $A(H)$. Moreover, $A(H)$ can be considered as an ideal of $B_\lambda(H)$, where $B_\lambda(H)$ is the dual of $C_\lambda^*(H)$.

A bounded linear operator on  Banach algebra ${\mathcal A}$ is called a right (resp. left) multiplier if it satisfies $ T(ab)= aT(b)$ (resp. $ T(ab) = T(a)b $)
for all $a, b\in{\mathcal A}$. We denote by $RM({\mathcal A})$ (resp.  $LM({\mathcal A})$) the space
of all right (resp. left) multipliers for ${\mathcal A}$. Clearly $ RM({\mathcal A}) $ and $LM({\mathcal A})$ are Banach algebras as  subalgebras of $B({\mathcal A})$, the space of all bounded linear operator on $ {\mathcal A} $. For any $ a \in \mathcal{A} $, let $ \rho_a: \mathcal{A}\longrightarrow \mathcal{A} $ (resp. $ \ell_a: \mathcal{A}\longrightarrow \mathcal{A}$)	be the  multiplication map defined by $ \rho_a (b)= ba $  (resp. $ \ell_a(b)= ab $) for all $ b \in \mathcal{A} $. Then $ \rho_a \in RM(\mathcal{A}) $ and  $ \ell_a \in LM(\mathcal{A}) $.
If $\mathcal A  $ is commutative, then $ RM({\mathcal A}) = LM({\mathcal A})$ and we denote it by $M({\mathcal A})  $.   An element $ a \in\mathcal A $ is called (weakly) completely continuous if $ \rho_a $  is a (weakly) compact operator on  $\mathcal A  $.  For the general theory of multipliers
we refer to Larsen \cite{lar}.

Let $ \mathcal{A} $ be a commutative Banach algebra. The    Arens products on  $ \mathcal{A}^{**} $ is defined as following three steps. For $ u, v$ in $\mathcal{A} $, $ \Phi $ in $ \mathcal{A}^* $ and  $ m,n \in \mathcal{A}^{**}, $ we define 
$ \Phi \cdot u, u \cdot \Phi $, 
$ m\cdot \Phi , \Phi \cdot m \in \mathcal{A}^* $
and
$ m \square n, m \diamondsuit n \in \mathcal{A}^{**} $
as follows:
\begin{align*}
\langle \Phi  \cdot u , v\rangle&=  \langle \Phi  , uv\rangle,\qquad \langle u \cdot \Phi  , v\rangle=  \langle \Phi  , uv\rangle \\ 
\langle m \cdot \Phi  , u \rangle &= \langle m , \Phi  \cdot u \rangle,\quad \langle \Phi  \cdot m  , u\rangle = \langle m , u \cdot \Phi  \rangle\\
\langle m\square n , \Phi   \rangle &= \langle m\ , n \cdot \Phi   \rangle,\quad \langle  m \diamondsuit n  , \Phi \rangle = \langle m , \Phi  \cdot n \rangle.
\end{align*}

$ \mathcal{A} $ is said to be Arens regular if $ \square $ and $ \diamondsuit $ coincide on $ \mathcal{A}^{**} $.
For any $ m \in \mathcal{A}^{**} $ the mapping $ n\mapsto m\square n $ is weak$ ^*$-weak$ ^* $ continuous on $ \mathcal{A}^{**}  $. However, the mapping $ n\mapsto n\square m $ need not to be  weak$ ^*$-weak$ ^* $-continuous.  The left topological center  of  $ \mathcal{A}^{**} $ is defined as 
\begin{align*}
\mathcal{Z}(\mathcal{A}^{**} , \square) = \{&m \in \mathcal{A}^{**} :\text{The mapping}~ n\mapsto n\square m  \\
&\text{is weak}^{*}\text{-weak} ^{*} ~\text{continuous on }  \mathcal{A}^{**} \}.
\end{align*}

A linear functional
$ m\in VN(H)^{*} $
is called a topologically  invariant mean 
on $VN(H)$
if  $\lVert  m \rVert =\langle m,     \lambda({\dot{e}}) \rangle  =1$
and  $ \langle m, u \cdot \Phi \rangle = u(\dot{e}) \langle m,\Phi  \rangle $ 
for every
$ \Phi \in VN(H)  $, 
$ u \in A(H).$
We denote by $TIM(\widehat{H})$ the set of all topologically invariant means on $VN(H)$.  It has been shown by Kumar \cite{kum1} that
$ VN(H) $ always admits a topologically invariant mean.

\section{$ Soc (A(H)) $ and discreteness of $ H $}
In this section, we will give some characterizations of discreteness of an  ultraspherical hypergroup $ H $  in the terms of the  algebraic properties of $A(H)  $. 
Recall that in a commutative Banach algebra $ \mathcal{A}  $ a non-zero element $ e $ satisfying $ e^{2} =e $ and $ e\mathcal{A} = \mathbb{C}e $ is called minimal idempotent.
If $ \mathcal{A} $  has minimal ideals, the smallest ideal containing all of them is called Socle of  $ \mathcal{A} $ and is denoted by Soc$(\mathcal{A})  $. If $ \mathcal{A} $ does not have minimal ideals, we define Soc$(\mathcal{A}) = \{0\} $.

Let $ H $ be a hypergroup. Then the set
$$G(H):= \{ \dot{x}\in H; \delta_{\dot{x}}\ast\delta_{\check{\dot{x}}}=\delta_{\dot{e}}=\delta_{\check{\dot{x}}}\ast\delta_{\dot{x}}\}.$$
is a locally compact group and   is called maximum subgroup of $ H $.  In what follows, $H$ will always be an ultraspherical hypergroup associated to  a locally compact group
$G$ and a spherical
projector $\pi: C_c(G)\rightarrow C_c(G)$.

\begin{proposition}\label{5.5}
	$ H $ is discrete if and only if there exists a minimal idempotent $u\in A(H) $ such that $ u(\dot{a})\neq0 $ for some $ \dot{a} \in G(H) $. In this case $u= {\bf1}_{\dot{a}} $,  where ${\bf1}_{\dot{a}}$ denote the characteristic
	function at $\{\dot{a}\}$ on $H$.
	
\end{proposition}
\begin{proof}
	If $ H $ is discrete, then it
	suffices to take $ u=  {\bf1}_{\dot{e}}. $ 
	Conversely, let $ u $ be a minimal idempotent in $ A(H) $   $u\in A(H) $ such that $ u(\dot{a})\neq0 $ for some $ \dot{a} \in G(H) $. Then by \cite[Proposition IV.31.3]{bon}  and using the commutativity of $ A(H) $,   there is $ \Phi  \in A(H)^{*} $ such that $ vu = \Phi (v)u $ for all $ v \in A(H). $
	Now, we show that $ \Phi =\Phi_{\dot{x}} $ for some $ \dot{x}\in H$, where $\Phi_{\dot{x}}(v)=v(\dot{x})$ for all $v\in A(H)$.
To prove this, note that for each $v, w \in A(H) $, we have
	$$ \Phi (vw)u= (vw)u= (vu)(wu)=\Phi (v)\Phi (w)u,$$
	which implies that $ \Phi$ is a multiplicative functional.  Hence by \cite[ Theorem 3.13]{mur2} there is $ \dot{x}\in H $ such that $ \Phi =\Phi_{\dot{x}}$. It follows that $ vu= v(\dot{x})u $ for all $v \in A(H)$ and in particular $ u(\dot{x})=1$.  Now, we show that $ u={\bf 1}_{\dot{x}}. $  Let $ \dot{t}\in H-\{\dot{x}\} $ and choose $ v\in A(H) $ such that $ v(\dot{x}) \neq v(\dot{t}). $ Then we have
	$$v(\dot{x})u(\dot{t})=vu(\dot{t})=v(\dot{t})u(\dot{t}),$$
	which implies that $ u(\dot{t})=0. $ Since $ u(\dot{a})\neq0, $   it follows that $ \dot{x}=\dot{a} $. Therefore,  $ u={\bf 1}_{\dot{a}} $ and $ H $ must be discrete.   
\end{proof}

For an  ideal $ I $  of $ A(H) $, we denote by  $  Z(I) $ the set of all $ \dot{x} \in H $ such that $ v(\dot{x})=0 $ for all $ v \in I $. 
\begin{corollary}\label{5.6a}
	There is a minimal ideal $ \mathfrak{m} $ in $ A(H) $  with $ G(H)\nsubseteq Z(\mathfrak{m}) $ if and only if $ H $ is discrete and $  \mathfrak{m} = \mathbb{C}{\bf1}_{\dot{a}} $ for some $ \dot{a}\in G(H) . $ 
\end{corollary}     
\begin{proof}
	Let $ \mathfrak{m} $ be a minimal ideal in $ A(H)$ such that $ G(H)\nsubseteq Z(\mathfrak{m}) $. Then by  \cite[Proposition IV.30.6]{bon} there is a minimal idempotent $  u $ in $ A(H) $ such that $ \mathfrak{m}= uA(H)$. Since $ G(H)\nsubseteq Z(\mathfrak{m}) $, it follows from Proposition \ref{5.5} that $ H $ is discrete and $ u={\bf 1}_{\dot{a}}$ for some $\dot{a}\in G(H)$. Thus, $ \mathfrak{m}= \mathbb{C}{\bf 1}_{\dot{a}}$.
	Conversely, if $ H $ is discrete, then $  \mathfrak{m} = \mathbb{C}{\bf1}_{\dot{e}} $ is the desired minimal ideal in $ A(H). $
\end{proof}
\begin{proposition}\label{3.1}
	The following conditions are equivalent.
	
	{\rm(i)} $ H $ is discrete.
	
	{\rm(ii)} There is  $ u \in Soc(A(H))$ such that $ u(G(H))\neq\{0\} $.
	
	{\rm(iii)} There is  $u\in Soc(B_\lambda(H))$ such that $ u(G(H))\neq\{0\}  $.
	
\end{proposition}

\begin{proof}
	(i)$ \Leftrightarrow $(ii). By Corollary \ref{5.6a}, discreteness of $ H $ is equivalent to the existence of a minimal ideal $ \mathfrak{m} $ in $ A(H)$ with  $ G(H)\nsubseteq Z(\mathfrak{m}) $,   which is equivalent to the existence of $ u \in Soc(A(H)) $ with $ u(G(H))\neq\{0\} $.
	
	(ii)$ \Leftrightarrow $(iii). Since  $ A(H) $ is an ideal in $  B_\lambda(H)$ and separates the points of $ H $, it follows from  \cite[Proposition 1]{gha} that $Soc(A(H))= Soc(B_\lambda(H))$.
\end{proof} 

\begin{proposition}
	Suppose that $ A(H) $ admits a bounded approximate identity.  Then $ H $ is finite if and only if  $Soc(A(H)^{**})= A(H)^{**}$. 
\end{proposition}
\begin{proof}
	If $ H $ is finite, then $ A(H) $ is finite dimensional. Therefore,
	$$ Soc(A(H)^{**}) =Soc(A(H)) = A(H)= A(H)^{**}.$$
	The converse is an immediate consequence of  \cite[Theorem 1]{gha}. 
\end{proof}

\section{ Weakly compact multipliers of $ A(H) $} 
In the group setting Lau in \cite{lau1979}  proved that a locally compact group $ G $ is discrete if and only if $ A(G) $ admits (weakly) completely continuous elements. Among other things, we extend this result to the ultraspherical hypergroup setting.
The proof of the following lemma is an adaptation of
the proof given in \cite[ Lemma 4.7]{lau1981}.

\begin{lemma}\label{4.7} 
	Let $ H $ be a non-discrete  ultraspherical hypergroup. Then $ \lambda(\dot{x}) \notin C^{*}_{\lambda}(H) $ for all $ \dot{x} \in G(H). $
\end{lemma} 
\begin{proof}
	Suppose, in contrary, that  $ \lambda(\dot{x})\in C^{*}_{\lambda}(H) $ for some $ \dot{x} \in G(H)$. Then $ \lambda({\check{\dot{x}}})\notin C^{*}_{\lambda}(H)  $ by \cite[ Corollary 4.4]{esn1}. Therefore, by Hahn-Banach theorem there is $ m\in VN(H)^* $ such that $ \langle m, \Psi \rangle=0 $ for all $ \Psi \in C^{*}_{\lambda}(H)  $ and  $ \langle m,\lambda(\check{\dot{x}})\rangle\neq0. $ Let $ n\in VN(H)^*$ be defined by $\langle n, \Phi \rangle= \overline{\langle m, \Phi ^*\rangle}$ for all $\Phi \in VN(H)$. Then  $ \langle n,\lambda(\dot{x})\rangle =\overline{\langle m, \lambda (\check{\dot{x}})  \rangle}\neq0$ and $   \langle n,\Psi \rangle =0 $ for all $ \Psi \in C^{*}_{\lambda}(H)$, which is impossible.
\end{proof} 
\begin{proposition}\label{4.5}
	Let  $ u \in A(H). $ Then $ \rho_u $ is weakly compact if and only if $ \rho_u^{*}(VN(H))\subseteq C^{*}_{\lambda}(H). $
\end{proposition}
\begin{proof}
	If $\rho_u  $ is weakly compact, then $\rho_u^{*} $ is weak$ ^* $-weak continuous. Suppose that $ \Phi \in \rho_u^{*}(VN(H))$. Since $ C^{*}_{\lambda}(H) $ is weak$ ^*$ dense in $ VN(H) ,$ there is a net $ (\Psi_\alpha) $ in $ C^{*}_{\lambda}(H) $ such that  weak-$\lim\limits_{\alpha} \rho_u^{*}(\Psi_\alpha) =\rho_u^{*}(\Phi).$ Hence, $ \rho_u^{*}(VN(H))\subseteq \overline {\rho_u^{*}(C^{*}_{\lambda}(H))}^w\subseteq C^{*}_{\lambda}(H)$.
	
	Conversely, let $ \rho_u^{*}(VN(H))\subseteq C^{*}_{\lambda}(H) $. Suppose that $ (\Phi_\alpha) $ is a net in $ VN(H) $ such that weak$ ^* $-$\lim\limits_{\alpha} \Phi_\alpha =\Phi$ for some $ \Phi \in VN(H). $ Now, let $ m\in VN(H)^* $ and $ \varphi \in B_\lambda(H) $ be the restriction of $ m$ on  $ C^{*}_{\lambda}(H)$. Using the fact that  $ A(H) $  is an ideal in $ B_\lambda(H) $, we have
	\begin{align*}
	\langle u\cdot \Phi_{\alpha}, m\rangle
	=\langle u\cdot \Phi_{\alpha}, \varphi\rangle
	&=\langle \Phi _{\alpha}, u \varphi \rangle\\
	\longrightarrow
	\langle \Phi , u \varphi\rangle 
	& =\langle u\cdot \Phi , \varphi \rangle
	=\langle u\cdot \Phi , m\rangle.
	\end{align*}
	Hence,  $ \rho_u^{*}$ is  weak$ ^* $-weak continuous. It follows from \cite[ Theorem 3.5.14]{meg} that $ \rho_u $  is weakly compact.
\end{proof}

\begin{theorem}\label{4.1}
	The following conditions are equivalent.
	
	$  (\rm{i}) $  $ H $ is discrete.
	
	$ (\rm{ii}) $	$ \rho_u $ is compact for every $ u \in A(H). $
	
	$ (\rm{iii}) $ There exists  $ u \in A(H) $ such that $ \rho_u $ is weakly compact
	and  $  u(\dot{a})\neq0 $ for
	
	\ \ \ \ \ some  $ \dot{a} \in G(H)$.
\end{theorem}
\begin{proof}
	$ (i)\Rightarrow (ii)$. If $ H $ is discrete and $ u= 
	{\bf{1}}_{\dot{a}} $ for some $ {\dot{a}}\in H $, then $ \rho_{u}(A(H))=  \{\lambda {\bf{1}}_{\dot{a}} : \lambda \in \mathbb{C} \}.$ This implies that $ \rho_{u} $ is compact.
	Therefore, $ \rho_{u} $ is compact  for every $ u\in A(H)  $ with finite support. Now, since the set  of  all $ u\in A(H) $ such that $ u $ has finite support, is dense in $ A(H)$, a simple approximation argument
	gives that $ \rho_u $ is compact for all $ u\in A(H). $
	
	$ (ii)\Rightarrow (iii)$. This is obvious.

	$ (iii)\Rightarrow (i)$. Let $ u \in A(H)$ be such that  $  u(\dot{a})\neq0 $ for some  $ \dot{a} \in G(H)$ and  $ \rho_u $ is weakly compact. Then for each $ v\in A(H), $ we have
	$$\langle \rho_u^{*}(\lambda(\dot{a})), v\rangle= \langle \lambda(\dot{a}), \rho_u(v)\rangle= u(\dot{a}) \langle  \lambda(\dot{a}),v\rangle. $$
	Since $  u(\dot{a})\neq0 $, it follows that  $\lambda(\dot{a})\in \rho_u^{*}(VN(H))\subseteq C^{*}_{\lambda}(H). $ Hence, $ H $ is discrete by Lemma \ref{4.7}.
\end{proof}

It is known that a Banach algebra $ \mathcal{A} $ is an ideal in $ \mathcal{A}^{**} $ if and only if multiplication
operators in ${\mathcal A}$ are weakly compact; see \cite[page 248]{dal}. The next result generalizes   \cite [Theorem 3.7]{lau1981}.

\begin{corollary}\label{3.6}
	$ H $ is discrete if and only if $ A(H) $ is an ideal in $ VN(H)^* $.
\end{corollary}  

Let $H$ be  a locally compact group. Since $ G(H)=H $, the following result is an immediate consequence of Theorem \ref{4.1}.

\begin{corollary}
	A locally compact group $ G $ is discrete if and only if there is a non-zero $ u \in A(G) $ such that $ \rho_u $ is weakly compact.
\end{corollary}

Since $ G(H) $ always possesses $ \dot{e} $, we have the following corollary.
\begin{corollary}\label{3.7}
	$ H $ is discrete if and only if  there is   $ u \in A(H) $ with $ u(\dot{e})\neq0 $ such that $ \rho_u $ is weakly compact.
\end{corollary}

\begin{corollary}
	The following conditions are equivalent.
	
	{\rm(i)} $ H $ is discrete.
	
	{\rm(ii)} $ A(H) $ has a 1-dimensional ideal $ I $ such that $ G(H)\nsubseteq Z(I) $.
	
	{\rm(iii)} $ A(H) $ has a  finite dimensional ideal $ I $ such that $ G(H)\nsubseteq Z(I) $.

\end{corollary}  
\begin{proof}
	(i)$ \Rightarrow $(ii). Suppose that $ H $ is discrete. Then  for each $ \dot{a} \in G(H) $ the set $\mathbb{C}{\bf{1}}_{\dot{a}} $ is a non-zero  1-dimensional ideal in $ A(H). $ 
	
	(ii)$ \Rightarrow $(iii). This is clear.
	
	(iii)$ \Rightarrow $(i). Let $ I $ be a  finite dimensional ideal in $ A(H) $ with  $ G(H)\nsubseteq Z(I) $. Then there is   $ u \in I  $ such that $ u(G(H))\neq\{0\} $ and $\rho_u $ has finite rank. Therefore, Theorem \ref{4.1} implies that $ H$ is discrete.  
\end{proof}

\begin{corollary}\label{4.8}
	$ H $ is discrete if and only if there exists  $ u \in B_\lambda(H) $ such that $ \rho_u $ is weakly compact  on $B_\lambda(H) $ and  $  u(\dot{a})\neq0 $ for some  $ \dot{a} \in G(H)$.
\end{corollary}

\begin{proof}
	Let $ v \in A(H) $ be such that $v(\dot{a})\neq0 $. Then $ \rho_{vu}$ is weakly compact on $ A(H) $. Hence, $ H $ is discrete by Theorem \ref{4.1}. For the converse, choose $ u = {\bf 1}_{\dot{e}} $.
\end{proof}

\begin{proposition}\label{5.3}
	Let $ H $ be an ultraspherical hypergroup on amenable locally compact
	group $ G $ and let $ u \in B_\lambda(H) $. If $ \rho_u $ is weakly compact, then $ u \in A(H) $. 
\end{proposition}

\begin{proof}
	Define the map $ \widetilde{\rho}_u: C^*_{\lambda}(H)\longrightarrow C^*_{\lambda}(H)$  by $ \widetilde{\rho}_u(\Psi)= u\cdot \Psi$ for all $\Psi\in C^*_\lambda(H)$. It is now easily verified that $ \widetilde{\rho}_u^{*}=\rho_u $. Hence, $ \rho_u $ is weak$ ^* $-weak continuous. Using arguments similar to those in the proof of Proposition \ref{4.5}, replacing $C^*_{\lambda}(H)  $ by $ A(H) $ and $ VN(H) $ by $ B_\lambda (H) $, one can show that  the range of $ \rho_u $ is contained in $ A(H) $.
	Now, since $ G $ is amenable, it follows from \cite[ Theorem 3.4]{esn1} that $ 1 \in B_\lambda (H) $. Therefore, $ u = \rho_u (1) \in A(H) $. 
\end{proof}

\begin{proposition}\label{4.10}
	$ H $ is discrete if and only if  there is  a weakly compact right multiplier $ T :VN(H)^*\longrightarrow VN(H)^*$ and $ m\in VN(H)^*$ such that $ T(m)\in A(H) $ and
	$ \langle T(m),\lambda(\dot{e})\rangle\neq0 $. 
\end{proposition}

\begin{proof}
	Suppose that $ H $ is discrete and consider $ u={\bf1}_{\dot{e}} \in A(H)$. Then the map $ \tilde{\rho}_u:VN(H)^*\longrightarrow VN(H)^* $ defined by $ \tilde{\rho}_u(m)=m\square u, $ is a weakly compact right multiplier of $ VN(H)^* $  with the desired properties.
	
	Conversely, first note that for each $ v \in A(H) $, we have
	$$ \rho_{T(m)}(v)= vT(m) = v\square T(m)= T(v\square  m)= T\circ \rho_m(v).$$
	Using this and the fact that the restriction of $ T $ on $ A(H) $ is weakly compact, we conclude that $ \rho_{T(m)}= T\circ \rho_m$ is weakly compact on  $ A(H)$. Now, since  $ \langle T(m),\lambda(\dot{e})\rangle\neq0$, $ H  $ must be discrete by Theorem \ref{4.1}. 
\end{proof}

\begin{remark}
	We note that the condition $ T(m)\in A(H) $	 can not be removed in Proposition \ref{4.10}. In fact, let $ m $ be a topologically invariant mean on $  VN(H)$. Then, the map $T:VN(H)^*\longrightarrow VN(H)^*$ defined by $  T (n) = n \square m =\langle n , \lambda(\dot{e})\rangle m$ is a rank one right multiplier of $ VN(H)^* $, and hence is weakly compact.
\end{remark}

\begin{proposition}\label{6.11}
	Let $ H $ be an ultraspherical hypergroup on amenable locally compact
	group $ G $ and let $ \Lambda: VN(H)\longrightarrow  C^{*}_{\lambda}(H) $ be a $ weak^*$-$weak^* $ continuous $ A(H) $-module homomorphism. Then $ \Lambda=\rho^*_{u} $ for some $ u \in A(H) $.
\end{proposition}

\begin{proof}
	Since $ \Lambda $ is weak$^*$-weak$^* $ continuous,  by  \cite[Theorem 3.1.11]{meg}, there is a bounded linear operator $ S : A(H) \longrightarrow A(H) $ such that $ S^{*} =\Lambda $. It is easy to see that $  S $ is a multiplier of $ A(H) $. As $ G $ is amenable, we obtain from  \cite[ Theorem 3.4]{esn1} that $  S = \rho_u $ for some $ u \in B_{\lambda}(H) $. Now, since $ \rho_u^{*}(VN(H)) =\Lambda(VN(H))\subseteq C^{*}_{\lambda}(H) $, using arguments similar to those in the proof of Proposition \ref{4.5}, one can see that $ \rho_u $ is weakly compact. Therefore, Proposition \ref{5.3} implies that $ u \in A(H) $.
\end{proof}

\begin{proposition}\label{6.12}
	Let $ H $ be a discrete ultraspherical hypergroup on amenable locally compact
	group $ G $ and let $ \Lambda: VN(H)\longrightarrow  C^{*}_{\lambda}(H) $ be a bounded $ A(H) $-module homomorphism. Then $ \Lambda=\rho^*_{u} $ for some $ u \in A(H) $.
\end{proposition}

\begin{proof}
	Suppose that $ H $ is discrete. Then by Proposition \ref{4.1}, the multiplication in $ A(H) $ is weakly compact. Since $ A(H) $ is a weakly  sequentially complete Banach algebra which admits a bounded approximate identity, it follows from \cite[Theorem 3.1]{blp} that  $ \Lambda=\rho^*_{u} $ for some $ u \in A(H) $.
\end{proof}

Next, we show that $ \mathcal{Z}(A(H)^{**} , \square) $ is invariant under the weak$ ^*$-weak$ ^* $  continuous elements  of $LM (A(H)^{**} , \square) $. 
For any $ m \in A(H)^{**} $, let
$ L_m: A(H)^{**}\longrightarrow A(H)^{**} $ be the left multiplication operator defined by $L_m(n) = m \square n $.

\begin{proposition}\label{5.7}
	Let  $ T $ be a   weak$ ^*$-weak$ ^* $  continuous element  of $LM (A(H)^{**} , \square) $. Then $T (\mathcal{Z}(A(H)^{**} , \square))\subseteq \mathcal{Z}(A(H)^{**} , \square)  $. 
\end{proposition}

\begin{proof}
	Let $ m \in \mathcal{Z}(A(H)^{**} , \square) $. Then $ L_m $ is weak$ ^*$-weak$ ^* $  continuous. It is easy  to see that $L_{T(m)}=T\circ L_m  $, hence the mapping $ L_{T(m)} $ is weak$ ^*$-weak$ ^* $ continuous. Therefore, $ T(m)\in \mathcal{Z}(A(H)^{**} , \square) $.
\end{proof}

\begin{corollary}\label{5.8}
	Let $ T $ be a   weak$ ^*$-weak$ ^* $-continuous weakly compact   element  of  $LM (A(H)^{**} , \square) $. Then $T (A(H)^{**} , \square)\subseteq \mathcal{Z}(A(H)^{**} , \square)  $.
\end{corollary}

\begin{proof}
	
	Since $ T $ is  weak$ ^*$-weak$ ^* $ continuous and weakly compact, it follows from \cite[Theorem 3.1.11,  Theorem 3.5.14]{meg} that $ T $ is  weak$ ^*$-weak  continuous. Using the weak$ ^*$ density of $ \mathcal{Z}(A(H)^{**} , \square) $ in $A(H)^{**}  $, we conclude that  $ T(A(H)^{**}  )\subseteq T(\mathcal{Z}(A(H)^{**}, \square )^{-w} $. Hence, by  Proposition \ref{5.7} and Mazur's theorem, $T (A(H)^{**} , \square)\subseteq \mathcal{Z}(A(H)^{**} , \square)  $.
\end{proof}

\section{Arens regularity of closed ideals }
In this section, we study  Arens regularity of the closed ideals of $ A(H) $. 
We obtain results similar to those of Forrest \cite{for0} obtained in the group setting.

\begin{proposition}\label{f1}
	Let $ I $ be a closed ideal in $ A(H) $ such that $ \dot{e} \notin Z(I)$. If $ I $ is Arens regular, then $ VN(H) $ has a unique topologically invariant mean. In particular, $ H $ is discrete.
\end{proposition}
\begin{proof}
	Let $ m \in TIM(\widehat{H}) $ and let $ \Phi \in I^{\perp}. $ Take $ u\in I $ with $ u(\dot{e})=1. $ Then for each $ v \in A(H) $, we have
	$$\langle u\cdot \Phi , v\rangle=\langle  \Phi , uv\rangle= 0. $$
	Therefore, $ u\cdot \Phi =0, $  and then
	$$\langle m , \Phi \rangle=u(\dot{e}) \langle  m , \Phi \rangle= \langle  m ,u\cdot \Phi \rangle = 0. $$
	Hence, $ m \in I^{\perp \perp}$. Identifying $ I^{\perp \perp} $ with $ I^{**} $,  we get that $ TIM(\widehat{H}) \subseteq I^{**}$.
	Suppose  now that $ m_{1}, m_{2} \in TIM(\widehat{H}). $ Then  
	$$\langle  m_{2}\square \Phi , v\rangle=  v(\dot{e})\langle m_{2}, \Phi \rangle =\langle m_{2}, \Phi \rangle \langle \lambda(\dot{e}) , v \rangle \quad(v \in A(H)). $$
	Hence, for each $\Phi \in VN(H), $ we have
	$$ \langle m_{1}\square m_{2}, \Phi \rangle=\langle m_{1} ,m_{2} \square \Phi \rangle= \langle m_{1}, \lambda(\dot{e}) \rangle\langle m_{2} \square \Phi \rangle= \langle m_{2} , \Phi \rangle.$$
	This shows  that $ m_{1}\square m_{2}=m_{2}. $
	Finally, since $ I $  is Arens regular, $I^{**}$ is commutative and hence,
	$$m_{2} = m_{1}\square m_{2}=m_{2}\square m_{1}=m_{1}.$$ 
	The last statement  follows from \cite[ Theorem 1.7]{kum2}.
\end{proof}

\begin{proposition}\label{f2}
	Let $ I $ be a  closed ideal in $ A(H) $ with $ \dot{e} \notin Z(I)$. If $ I $ has a bounded approximate identity, then $ I  $ is Arens regular if and only if $ I $ is reflexive.
\end{proposition}
\begin{proof}
	If $ I $ is reflexive, then $ I $ is obviously Arens regular.
	
	For the converse, suppose that  $ I $ is Arens regular. Then, it follows from Proposition \ref{f1} that $ H $ is discrete.
	Since $ I $ has a bounded approximate identity, by  Cohen's
	factorization theorem,  $ I=I^2. $ Let $ \iota: A(H) \longrightarrow A(H)^{**}$ be the natural embedding. Then, we have
	$$\iota(I) \circ I^{**} = \iota(I \cdot I) \circ I^{\perp \perp}\subseteq \iota(I) \circ (\iota(I)\circ A(H)^{**})\subseteq I \cdot A(H) \subseteq I.$$
	Therefore, $ I $ is an ideal in $ I^{**} $. On the other hand, since $ A(H) $ is weakly sequentially complete, $ I $ is also weakly sequentially complete. It follows from  \cite[ Theorem 2.9.39]{dal} that $ I $ is unital.  The inclusion $ I \subseteq C_0(H) $ and discreteness of $ H $ implies that $ I $ is finite dimensional and hence $ I $ is reflexive.
\end{proof}

In the following result, we denote the set of all weakly compact multipliers of $ A(H)$  by $ M_{wc}(A(H)) $. 
\begin{corollary}
	Let $ H $ be an ultraspherical hypergroup on amenable locally compact
	group $ G $. Then  $ M_{wc}(A(H))= B_\lambda(H) $ if and only if $ H $ is finite. 
\end{corollary}

\begin{proof}
	Suppose that $ M_{wc}(A(H))= B_\lambda(H) $. Since $ A(H) \subseteq B_\lambda(H) $, we obtain from Theorem \ref{4.1} that $ H $  is discrete. On the other hand,  $ M(A(H))= B_\lambda(H) $ by \cite[Theorem 3.4]{esn1},  which implies that $ 1 \in B_\lambda(H) $ and hence the set $\{ u \cdot 1 : \lVert u\rVert_{A(H)}\leqslant 1\} $ in $A(H)$  is weakly compact. This means that $ A(H) $  is reflexive. Therefore,  $ A(H) $ is Arens regular by Proposition \ref{f2}. This and   \cite[Theorem 2.9.39]{dal} imply that $ A(H) $ is unital. Finally, since $ A(H) \subseteq C_0(H) $ separates the points of  $ H $, we conclude that $ H $ is finite. The converse is trivial.
\end{proof}

We end this paper with the following questions.\\


1-If $ T :VN(H)^*\longrightarrow VN(H)^*$ is a weakly compact left multiplier such that $ \langle T(n) , \lambda(\dot{e})\rangle \neq 0$ for some $ n \in VN(H)^* $, must $H$ be discrete? In the group setting, Ghahramani and Lau in \cite[Theorem 4.3]{ghl} have given a positive  answer to this question.\\

2-If  there is a non-zero $ u \in A(H) $ such that $ \rho_u $ is (weakly)  compact, must $H$ be discrete?


\bibliographystyle{amsplain}

\end{document}